\documentclass[times,twoside,a4,11pt]{amsart}

\usepackage[english]{babel}
\usepackage[utf8x]{inputenc}
\usepackage{url}
\usepackage{amsmath}
\usepackage{graphicx}
\usepackage[colorinlistoftodos]{todonotes}
\usepackage{amsmath,amscd,amsthm}
\usepackage{amstext, amsfonts, a4}
\usepackage{amssymb}
\usepackage{tikz-cd}
\usepackage{fancyhdr}
\usepackage{enumerate}
\usepackage{cleveref}
\usepackage{xcolor}

\numberwithin{equation}{section}

\newtheorem{prop}[equation]{Proposition}
\newtheorem{cor}[equation]{Corollary}
\newtheorem{lem}[equation]{Lemma}

\theoremstyle{definition}
\newtheorem{exmp}[equation]{Example}

\newtheorem{rem}[equation]{Remark}

\theoremstyle{plain}

\renewcommand{\phi}{\varphi}
\renewcommand{\dim}{\operatorname{\mathsf{dim}}}

\renewcommand{\setminus}{\smallsetminus}

\newcommand\Br{\operatorname{\mathsf{Br}}}

\newcommand\Frob{\operatorname{\mathsf{Frob}}}

\newcommand\CH{\operatorname{\mathsf{CH}}}

\newcommand{\llangle}{\langle\!\langle}

\renewcommand{\leq}{\leqslant}
\renewcommand{\geq}{\geqslant}

\begin{document}
\title{Mixed multiquadratic splitting fields}
	
\date{31 March, 2025}

\author{Fatma Kader B\.{i}ng\"{o}l}
\author{Adam Chapman}
\author{Ahmed Laghribi}

\address{Scuola Normale Superiore, Piazza dei Cavalieri 7, 56126 Pisa, Italia}
\email{fatmakader.bingol@sns.it}

\address{School of Computer Science, Academic College of Tel-Aviv-Yaffo, Rabenu Yeruham St., P.O.B 8401, Yaffo, 6818211, Israel}
\email{adam1chapman@yahoo.com}

\address{Univ. Artois, UR 2462, Laboratoire de Math{\'e}matiques de Lens (LML), F-62300 Lens, France}
\email{ahmed.laghribi@univ-artois.fr}

\begin{abstract}
    We study mixed multiquadratic field extensions as splitting fields for central simple algebras of exponent $2$ in characteristic $2$.
    As an application, we provide examples of nonexcellent mixed biquadratic field extensions.
		
    \medskip\noindent
    {\sc Classification} (MSC 2020): 11E04, 11E81, 16K20, 13A35
		
\end{abstract}
	
\maketitle

\section{Introduction}
Throughout this short note, $F$ denotes a field of characteristic $2$. 
Let $\Br_2(F)$ denote the $2$-torsion subgroup of the Brauer group of $F$. 
It is well-known that $\Br_2(F)$ is generated by classes of $F$-quaternion algebras, see \cite[Chap. VII, \S 9]{Alb39}.
Recall that an $F$-quaternion algebra is a central simple $F$-algebra of degree $2$.  
For $a\in F$ and $b\in F^{\times}$, we denote by $[a,b)_F$ the $F$-quaternion algebra generated by two elements $i$ and $j$ satisfying the relations 
\[i^2+i=a,\ j^2=b\mbox{ and }ji=(1+i)j.\]
The quadratic extension $F(i)$ of $F$ contained in $[a,b)_F$ is a separable extension, while $F(j)$ is inseparable.
Any $F$-quaternion algebra is isomorphic to $[a,b)_F$ for certain $a\in F$ and $b\in F^{\times}$.

Given a class $[A]$ in $\Br_2(F)$ of a central simple $F$-algebra $A$ of exponent at most $2$, we may ask for the smallest $m\in\mathbb{N}$ such that $[A]$ is equal to the class of a tensor product of $m$ $F$-quaternion algebras. 
This numerical invariant $m$ of $A$ is denoted by $\lambda_2(A)$ and called the $2$-symbol length of $A$. 
In \Cref{section:descent-quat-alg}, we study the behavior of this invariant under multiquadratic extensions of separability degree at most $8$ (\Cref{symbol-length-mixed-multi}). 
To this end we start with some descent results, in particular we prove \Cref{descent-quaternions-ins-quad} which is based on a transfer of central simple $K$-algebras induced by the usual Frobenius homomorphism when $K$ is an inseparable quadratic extension of $F$.  
The main result of this paper, which is given in \Cref{section:non-excellent-mixed-biquad-ext}, is the application of our methods to the excellent property for quartic extensions.

Recall that a field extension $K/F$ is called \emph{excellent} if for any quadratic form $\varphi$ over $F$, the anisotropic part $(\varphi_K)_{an}$ of  the quadratic form $\varphi_K$ obtained by extending scalars to $K$ is defined over $F$, i.e. there exists a quadratic form $\psi$ over $F$ such that $(\varphi_K)_{an}\simeq\psi_K$.
It is well know that, in arbitrary characteristic, any quadratic extension is excellent \cite[Lemma 5.4]{HL06}.
The same holds for degree $3$ extensions, more generally for odd degree extensions, in view of Springer's Theorem (\cite[Corollary 18.5]{EKM08}).
The situation changes for quartic extensions.
Nonexcellent separable biquadratic extensions exist, see \cite[\S 5]{ELTW83}, \cite{Siv04} for characteristic not $2$ and \cite[\S 6]{LD24} for characteristic $2$. 
In \cite{LD24}, the excellence property for inseparable quartic extensions is studied. 
They provide examples of simple inseparable quartic extensions which are nonexcellent.
Purely inseparable biquadratic extensions are known to be excellent, see \cite{Hoff15}.
The remaining case for inseparable quartic extensions is then the case of mixed biquadratic extensions.
In \Cref{section:non-excellent-mixed-biquad-ext}, we provide examples of nonexcellent mixed biquadratic extensions (\Cref{example}), hence completing the picture for the study of excellence property for quartic extensions.
Our examples are realized as subfields of non-decomposable cyclic algebras of degree $8$ and exponent $2$ defined over a field of characteristic $2$. 
Our method applies also to simple purely inseparable quartic extensions contained in such an algebra, hence it gives new examples of nonexcellent simple purely inseparable quartic extensions as well.

Our main references are \cite{EKM08} for the theory of quadratic forms, and \cite{Alb39} for the theory of central simple algebras.

\section{Descent of quaternion algebras}\label{section:descent-quat-alg}
For $a_1,\ldots,a_n\in F^{\times}$, let $\langle a_1,\ldots,a_n\rangle_{\mathfrak{b}}$ denote the non-degenerate symmetric bilinear form 
$$F^n\times F^n\to F,\, ((x_1,\ldots,x_n),(y_1,\ldots,y_n))\mapsto\sum_{i=1}^{n}a_ix_iy_i.$$

We will be only considering nonsingular quadratic forms (nondegenerate even-dimension quadratic forms in the terminology of \cite{EKM08}), meaning quadratic forms having polar form with trivial radical. 
Recall from \cite[Corollary 7.32]{EKM08} that any such form $\phi$ over $F$ decomposes as follows: 
$$\phi\cong [a_1,b_1]\perp\ldots\perp[a_n,b_n]$$ for some $a_1,b_1,\ldots,a_n,b_n\in F$, where for any scalars $a, b\in F$ we denote by $[a,b]$ the $2$-dimensional quadratic form $aX^2+XY+bY^2$. In particular, a nonsingualar quadratic form is of even dimension. 
The Arf invariant of the form $\phi$ is defined to be $\sum_{i=1}^{n}a_ib_i+\wp(F)$ in $F/\wp(F)$, where $\wp(F)=\{a^2+a\mid a\in F\}$.

Following \cite{EKM08}, we use the notation $I_qF$ for the Witt group of nonsingular quadratic forms over $F$, $I F$ for the ideal of even-dimensional forms in the Witt ring $W F$ of non-degenerate symmetric bilinear forms over $F$.
The group $I_qF$ is endowed with a $WF$-module structure induced by tensor product in a natural way. 
For $n\geq 2$ an integer, let $I^n_qF$ be the product $I^{n-1}F\otimes I_qF$. 
The Hauptsatz of Arason-Pfister \cite[Theorem 23.7]{EKM08} asserts that if the Witt class of an anisotropic quadratic form $\phi$ lies in $I^n_qF$, then $\dim \phi \geq 2^n$.

For $n\geq 2$, the quadratic $n$-fold Pfister form is a nonsingular quadratic form isometric to
$$\langle1,a_1\rangle_{\mathfrak{b}}\otimes\ldots\otimes\langle1,a_{n-1}\rangle_{\mathfrak{b}}\otimes[1,a_n]$$
where $a_1,\ldots,a_{n-1}\in F^{\times}$ and $a_n\in F$. This form is denoted by $\llangle a_1,\ldots, a_{n-1},a_n]]$. 
A quadratic $1$-fold Pfister form is isometric to $[1, a]$ for some $a\in F$.

Let us make the following observation relevant to quadratic $2$-fold Pfister forms for later use.
\begin{lem}\label{norm-2fold-insep-quad}
    Let $x,y,a,b\in F$ with $x^{2}+by^{2}\neq0\neq b$. 
    There exists some $c\in F$ such that $\llangle x^{2}+by^{2},a]]\simeq\llangle b,c]]$.
\end{lem}
\begin{proof}
    If $x=0$, then the statement holds with $c=a$. Assume now $x\neq0$. We have that
\begin{equation*}
\begin{split}		
    \llangle x^{2}+by^{2},a]]&=[1,a]\perp(x^{2}+by^{2})[1,a]\\
    &\simeq[x^{2},x^{-2}a]\perp[x^{2}+by^{2},(x^{2}+by^{2})^{-1}a]\\
    &\simeq[x^{2},x^{-2}a+(x^{2}+by^{2})^{-1}a]\perp[by^{2},(x^{2}+by^{2})^{-1}a]\\
    &\simeq[1,c]\perp b[1,d]
\end{split}	
\end{equation*}    
    where $c=a(1+x^{2}(x^{2}+by^{2})^{-1})$ and $d=aby^{2}(x^{2}+by^{2})^{-1}$. Note that $c+d=0$.
    It follows that $\llangle x^{2}+by^{2},a]]\simeq\llangle b,c]]$ with $c\in F$.
\end{proof}

For a finite field extension $K/F$ and an $F$-linear functional $s: K\to F$, 
the transfer of the quadratic form $\varphi$ (resp. a symmetric bilinear form $\mathfrak{b}$) over $K$ with respect to $s$ is denoted by $s_\ast(\varphi)$ (resp. $s_\ast(\mathfrak{b})$) \cite[\S 20.A]{EKM08}.
The functional $s: K\to F$ induces group homomorphisms $s:I^n_qK\to I^n_qF$ for any $n\geq1$; see \cite[Corollary 34.17]{EKM08}. 
We will need the following particular case of Frobenius reciprocity: For a non-degenerate symmetric bilinear form $\mathfrak{b}$ over $K$ and a nonsingular quadratic form $\phi$ over $F$, we have $s_\ast(\mathfrak{b}\otimes \phi)\cong s_\ast(\mathfrak{b})\otimes \phi$; see \cite[Proposition 20.2]{EKM08}. 

\begin{lem}\label{exact-r-s-I_q and I^2_q-quadext}
    Let $K/F$ be a quadratic field extension and $s:K\to F$ a nonzero $F$-linear functional such that $s(1)=0$. 
    Let $\varphi$ be an anisotropic even-dimensional quadratic form over $K$. Assume that $s_\ast(\varphi)$ is hyperbolic. 
    Then, there exists some nonsingular quadratic form $\psi$ over $F$ such that $\varphi\simeq \psi_K$.
    If furthermore $\varphi$ has trivial Arf invariant, then $\psi$ can be chosen to have trivial Arf invariant.
\end{lem}
\begin{proof}
    The first statement follows by \cite[Theorem 2.2]{BGBT18}.
    For the second statement, assume further that $\varphi$ has trivial Arf invariant. 
    By the first part, there exists a nonsingular quadratic form $\psi$ over $F$ such that $\varphi\simeq\psi_K$. 
	
    Let $d\in F$ be a representative of the Arf invariant of $\psi$. 
    Write $\psi\simeq\theta\perp[a,b]$ for a nonsingular quadratic form $\theta$ over $F$ and some $a,b\in F$.
    As $\varphi$ is anisotropic, we actually have $a,b\in F^{\times}$. 
    We set $\psi'=\theta\perp[a,b+\frac{d}{a}]$.
    Then $\psi'$ has trivial Arf invariant. 
    Note that $d\in\wp(K)$, hence by \cite[Example 7.23]{EKM08}, we have $$[a,b+\frac{d}{a}]_K\perp[a,b]_K\simeq[a,\frac{d}{a}]_K\perp[0,0]\simeq a[1,d]_K\perp[0,0]\simeq[0,0]\perp[0,0].$$ 
    Consequently, $[a,b+\frac{d}{a}]_K\simeq[a,b]_K$. 
    It follows that $\psi'_K\simeq \psi_K\simeq\varphi$.
\end{proof}

Let $K/F$ be a purely inseparable field extension such that $K^2\subseteq F$. 
The Frobenius homomorphism $\Frob_{K/F}: K\to F$, $x\mapsto x^2$ induces a group homomorphism
$$\Frob_{K/F}: \Br_2(K)\to\Br_2(F), [A]\mapsto[\Frob_{K/F}A].$$
For $x\in K$ and $y\in K^{\times}$, we have that $\Frob_{K/F}([x,y)_{K})=[x^2,y^2)_{F}$. See \cite[\S 3]{B24} for details.

Recall furthermore from \cite[Theorem 14.3]{EKM08} the group homomorphism 
$$e_2: I^2_qF\to\Br_2(F)$$
which is defined  by mapping the Witt class of a quadratic form in $I^2_qF$ to the Brauer class of its Clifford algebra. 
The kernel of $e_2$ is $I^3_qF$ (see \cite[\S 16]{EKM08}).

\begin{lem}\label{comm-diag-e2-s-Frob-quadext}
    Let $K/F$ be an inseparable quadratic field extension and let $s:K\to F$ be a nonzero $F$-linear functional such that $s(1)=0$. 
    Then, the following diagram commutes: 
\begin{center}
\begin{tikzcd}
    I^2_qK\arrow{r}{s_\ast}\arrow{d}{e_2}&I^2_qF\arrow{d}{e_2}\\
    \Br_2(K)\arrow{r}{\Frob_{K/F}}&\Br_2(F)
\end{tikzcd}
\end{center}	
\end{lem}
\begin{proof}
    Using the isometry $[1,x]\simeq[1,x^2]$ for any $x\in K$ and the fact that $K^2\subseteq F$, it is clear that $I^2_qK=IK \otimes I_qF$. 
    Commutativity of the diagram follows by Frobenius reciprocity, the computation of transfers of $1$-fold Pfister bilinear forms in \cite[Lemma 34.19]{EKM08}, and the computation of the Frobenius of quaternion algebras in \cite[Proposition 3.2]{B24}.
\end{proof}

For central simple $F$-algebras $A$ and $B$, we write $A\sim B$ to indicate that $A$ and $B$ are Brauer equivalent.
\begin{prop}\label{descent-quaternions-ins-quad}
    Let $K/F$ be an inseparable quadratic field extension. Let $A$ be a central simple $K$-algebra of exponent $2$ and assume that $m\in\mathbb{N}_+$ is such that $A$ is Brauer equivalent to a tensor product of $m$ $K$-quaternion algebras. 
    Assume that $\Frob_{K/F}A=0$. 
    There exist $F$-quaternion algebras $H_1,\ldots,H_n$ where $n\leq2m-1$ such that $A\sim\bigotimes_{i=1}^{n}H_i\otimes_FK$.
\end{prop}
\begin{proof}
    Let $Q_1,\ldots, Q_m$ be $K$-quaternion algebras with $A\sim\bigotimes_{i=1}^{m}Q_i$.
    Let $a_i\in K$, $z_i\in K^{\times}$ be such that $Q_i\simeq[a_i,z_i)_K$ for $1\leq i\leq m$.
    Since $[a_i,z_i)_K\simeq[a_i^{2},z_i)_K$, we may assume that $a_i\in F$ for $1\leq i\leq m$.
    Fix $b\in F^{\times}\setminus F^{\times2}$ with $K=F(\sqrt{b})$. 
    Let $s:K\to F$ be the $F$-linear functional with $s(1)=0$ and $s(\sqrt{b})=1$. For $1\leq i\leq m$, write $z_i=x_{i}+y_{i}\sqrt{b}$ with $x_{i},y_{i}\in F$.
    Let $\lambda_i=y_{i}^{-1}$ if $y_{i}\neq0$, otherwise let $\lambda_i=1$.
    Finally we set $\varphi=\perp_{i=1}^{m}\lambda_i\llangle z_i,a_i]]$. Note that we have $e_2(\varphi)=[A]$.
    If $y_i=0$, then $s_\ast(\langle z_i\rangle)$ is clearly hyperbolic. If $y_i\neq0$, an easy computation shows that $s_\ast(\langle z_i\rangle)$ represents $y_i$ and that its determinant is $x_i^2+by_i^2$, hence $s_\ast(\langle z_i\rangle)\simeq\langle y_i,y_i(x_i^2+by_i^2)\rangle$.
    Now using Frobenius reciprocity, one computes that the Witt class of $s_\ast(\varphi)$ is given by $\perp_{i=1}^{m}\llangle x_{i}^{2}+y_{i}^{2}b,a_i]]$. 
    \Cref{norm-2fold-insep-quad} provide $c_1,\ldots,c_m\in F$ such that 
    $$\perp_{i=1}^{m}\llangle x_{i}^{2}+y_{i}^{2}b,a_i]]\simeq\perp_{i=1}^{m}\llangle b,c_i]]$$
    which is Witt equivalent to $\llangle b,\sum_{i=1}^{m}c_i]]$.
    Furthermore by \Cref{comm-diag-e2-s-Frob-quadext}, we have $$e_2(s_\ast(\varphi))=\Frob_{K/F}(e_2(\varphi))=\Frob_{K/F}A$$ which is trivial by our assumption. 
    Hence, the Witt class of $s_\ast(\varphi)$ lies in $I_q^{3}F$.
    Since $s_\ast(\varphi)$ is Witt equivalent to $\llangle b,\sum_{i=1}^{m}c_i]]$, this implies that $s_\ast(\varphi)$ is hyperbolic, by the Hauptsatz of Arason-Pfister.
    It follows by \Cref{exact-r-s-I_q and I^2_q-quadext} that $\varphi\simeq \psi_K$ for some nonsingular quadratic form $\psi$ over $F$ with trivial Arf invariant. 
    In particular, we have $\dim\psi=\dim\varphi=4m$.
    Then $e_2(\psi)=\sum_{i=1}^{2m-1}[H_i]$ for some $F$-quaternion algebras $H_1,\ldots,H_{2m-1}$ (see the proof of \cite[Lemma 38.1]{EKM08}).
    It follows that $$[A]=e_2(\varphi)=e_2(\psi_K)=\sum_{i=1}^{2m-1}[H_i\otimes_FK].$$
\end{proof}

As a corollary, we retrieve the case $p=2$ from \cite[Proposition 4.6]{B24}. 
\begin{cor}\label{compar-2-symbol-ins-quad}
    Let $K/F$ be an inseparable quadratic field extension. 
    Let $A$ be a central simple $F$-algebra of exponent $2$ and assume that $A_K$ is not split.
    Then $A\sim B\otimes_FH$ where $H$ is an $F$-quaternion algebra that splits over $K$ and $B$ is a central simple $F$-algebra with $\lambda_2(B)\leq2\lambda_2(A_K)-1$. 
    In particular $\lambda_2(A)\leq2\lambda_2(A_K)$.
\end{cor}
\begin{proof}
    Set $m=\lambda_2(A_K)$. We have by \cite[Proposition 3.2]{B24} that $\Frob_{K/F}A_K=0$. 
    Hence, \Cref{descent-quaternions-ins-quad} provide $F$-quaternion algebras $H_1,\ldots, H_{2m-1}$ such that $A_K\sim\bigotimes_{i=1}^{2m-1}H_i\otimes_FK$.
    In particular, $A\otimes_F\bigotimes_{i=1}^{2m-1}H_i$ is split over $K$, hence is Brauer equivalent to an $F$-quaternion algebra $H_{2m}$ that splits over $K$. 
    It follows that $A\sim\bigotimes_{i=1}^{2m}H_i$.
\end{proof}

\begin{prop}\label{symbol-length-mixed-multi}
    Let $A$ be a central simple $F$-simple algebra of exponent $2$.
    Assume that $A$ splits over $F(\alpha_1,\ldots,\alpha_m,\sqrt{b_1},\ldots,\sqrt{b_n})$ where $\alpha_i^{2}-\alpha_i\in F$ for $1\leq i\leq m$ and $b_1,\ldots,b_n\in F^{\times}\setminus F^{\times2}$. 
    Then $A\sim B\otimes_F\bigotimes_{i=1}^{n}[a_i,b_i)_F$ where $a_1,\ldots,a_n\in F$ and $B$ is a central simple $F$-algebra such that the following hold:
\begin{enumerate}[$(1)$]
    \item If $m=1$, then $\lambda_2(B)\leq1$, in particular $\lambda_2(A)\leq n+1$.
    \item If $m=2$, then $\lambda_2(B)\leq2^{n}+1$, in particular $\lambda_2(A)\leq2^{n}+n+1$.
    \item If $m=3$, then $\lambda_2(B)\leq3\cdot2^{n}+1$, in particular $\lambda_2(A)\leq3\cdot2^{n}+n+1$.
\end{enumerate} 
\end{prop}
\begin{proof}
    We prove the statement by induction on $n$. 
    In the case where $n=0$, $(1)$ is clear, and $(2)$ resp. $(3)$ follows by \cite{Alb32} resp. \cite{Row84}.
    We explain the proof of the statement only for the case where $m=3$; the argument is the same for each cases.

    Assume now $m=3$, $n\geq1$ and set $K=F(\sqrt{b_n})$.
    We have by induction hypothesis that $A_K\sim B'\otimes_K\bigotimes_{i=1}^{n-1}[a_i,b_i)_K$ for some $a_1,\ldots,a_{n-1}\in K$ and a central simple $K$-algebra $B'$ with $\lambda_2(B')\leq3\cdot2^{n-1}+1$. 
    Since $[a_i,b_i)_K\simeq[a_i^{2},b_i)_K$, we  may assume that $a_1,\ldots,a_{n-1}\in F$. 
    Hence $(A\otimes_{F}\bigotimes_{i=1}^{n-1}[a_i,b_i)_F)_K\sim B'$.
    It follows now by \Cref{compar-2-symbol-ins-quad} that $$A\otimes_{F}\bigotimes_{i=1}^{n-1}[a_i,b_i)_F\sim B\otimes_F[a_n,b_n)_F$$ for $a_n\in F$ and a central simple $F$-algebra $B$ with $\lambda_2(B)\leq2(3\cdot2^{n-1}+1)-1=3\cdot2^{n}+1$. 
    Therefore $A\sim B\otimes_F\bigotimes_{i=1}^{n}[a_i,b_i)_F$ and we have $\lambda_2(A)\leq3\cdot2^{n}+1+n$. 
\end{proof}

\begin{rem}
    Let $A$ be a central simple $F$-algebra of exponent $2$.
    Assume that $A$ splits over $F(\alpha_1,\ldots,\alpha_m,\sqrt{b_1},\ldots,\sqrt{b_n})$ where $\alpha_i^{2}-\alpha_i\in F$ for $1\leq i\leq m$ and $b_1,\ldots,b_n\in F^{\times}\setminus F^{\times2}$. 
    The bound in \Cref{symbol-length-mixed-multi} for $m=1$ coincides with the existing one; see \cite[Théorème 4]{MM95} also \cite[Proposition 4.2]{B24}. 
    For $m=2$, it is shown in \cite[Théorème 5]{MM95} that $\lambda_2(A)\leq n+3$, which can also be obtained from methods in \cite{B24}. 
    Finally for $m=3$, the methods in \cite{B24} applied to this case would give that $\lambda_2(A)\leq n+7$.
    These bounds are better than what we get in $(2)$ and $(3)$ of \Cref{symbol-length-mixed-multi} for $n\geq2$.
    They coincide with our bounds from \Cref{symbol-length-mixed-multi} for $n=1$.
\end{rem}

\section{Excellence property for biquadratic extensions}\label{section:non-excellent-mixed-biquad-ext}
In this section, we provide examples of nonexcellent mixed biquadratic extensions, completing the study of the excellence property of inseparable quartic extensions done in \cite{LD24}. 
The same argument also applies to establish the nonexcellence of simple purely inseparable quartic extensions.

\begin{prop}\label{mixed-biqud-non-decom-cyclic} 
    Let $A$ be a central simple $F$-algebra of degree $8$ and exponent $2$. 
    If $A$ contains a mixed biquadratic or simple purely inseparable quartic field extension of $F$ which is excellent, then $A$ decomposes into a tensor product of $3$ quaternion algebras over $F$. 
\end{prop}
\begin{proof}
    Let $M/F$ be an excellent quartic extension of $F$ contained in $A$ as in the statement, i.e. $M=F(\sqrt[4]{b})$ or $M=F(\alpha,\sqrt{b})$ where $\alpha^{2}-\alpha=a'\in F$ and $b\in F^{\times}$. Set $K=F(\sqrt{b})$. 
    Then $A_K\sim[a,x)_K\otimes_K[c,y)_K$ where $c\in F$, $x\in K^{\times}$, and $y=\sqrt{b}$, $a\in F$ when $M=F(\sqrt[4]{b})$ and $y\in K^{\times}$, $a=a'$ otherwise. 
    Let $s:K\to F$ be the $F$-linear functional with $s(1)=0$ and $s(\sqrt{b})=1$. 
    We choose $\lambda \in F$ such that $s_\ast(\langle x\rangle_{\mathfrak{b}})\perp \lambda s_\ast(\langle y\rangle_{\mathfrak{b}})$ is isotropic. 
    We set $\phi=\llangle x,a]]\perp\lambda\llangle y,c]]$ and observe that $e_2(\varphi)=[A_K]$. 
    Using Frobenius reciprocity, we compute that the Witt class of $s_\ast(\varphi)$ is given by $s_\ast(\langle x\rangle_{\mathfrak{b}})\otimes [1,a]\perp \lambda s_\ast(\langle y\rangle_{\mathfrak{b}})\otimes [1,c]$, which is isotropic,  
    since $s_\ast(\langle x\rangle_b)\perp \lambda s_\ast(\langle y\rangle_b)$ is isotropic.
    Therefore $\dim (s_\ast(\varphi))_{an}<8$.
    Now Lemma \ref{comm-diag-e2-s-Frob-quadext} implies that $$e_2(s_\ast(\varphi))=\Frob_{K/F}(e_2(\phi))=\Frob_{K/F}A_K$$
    which is trivial by \cite[Proposition 3.2]{B24}.
    Then, the Witt class of $s_\ast(\varphi)$ lies in $I^3_qF$.  
    As $\dim (s_\ast(\varphi))_{an}<8$, the Hauptsatz of Arason-Pfister implies that $s_\ast(\varphi)$ is hyperbolic. 
    Hence by \Cref{exact-r-s-I_q and I^2_q-quadext}, $\varphi\simeq \psi'_K$ for some nonsingular quadratic form $\psi'$ over $F$ with trivial Arf invariant. 
    Note that $\dim(\psi'_{M})_{an}=\dim(\varphi_{M})_{an}\leq4$. 
    Since $M/F$ is excellent, we have $(\psi'_M)_{an}\simeq \psi_M$ for some nonsingular quadratic form $\psi$ over $F$. 
    In particular $\dim\psi\leq4$. It is easy to see that we can choose $\psi$ with trivial Arf invariant. 
    Then $e_2(\psi)=[H]$ for some $F$-quaternion algebra $H$. 
    We have $$[H_M]=e_2(\psi_M)=e_2(\psi'_M)=e_2(\varphi_M)=[A_M].$$  
    It follows that $A\otimes_FH$ splits over $M$. 
    This implies that $A\otimes_FH$ is Brauer equivalent to a tensor product of $2$ $F$-quaternion algebras. 
    We deduce that $A$ decomposes into a tensor product of $3$ $F$-quaternion algebras.
\end{proof}

\begin{rem}
    Let $D$ be a cyclic $F$-division algebra of exponent $2$ and degree $8$. By \cite[Theorem 7.27]{Alb39}, $D$ contains a simple purely inseparable quartic field extension of $F$.
    
     Let further $K/F$ be an inseparable quadratic field extension $K/F$ contained in $D$. Then, the centralizer $C_D(K)$ of $K$ in $D$ is a central simple $K$-algebra that has exponent $2$ and degree $4$. 
    In particular $C_D(K)$ contains a separable quadratic field extension $L/K$, see \cite[Theorem 11.9]{Alb39}. Write $L=K(\alpha)$ with $\alpha\in L\setminus K$ and $\alpha^2-\alpha\in K$. 
    Then $L'=F(\alpha^2)$ is a separable quadratic field extension of $F$ and we have $L=KL'$. 
    In particular $L/F$ is a mixed biquadratic inseparable field extension contained in $D$.
\end{rem}

\begin{exmp}\label{example}
    Let $C$ be a cyclic algebra over a field $k$ of characteristic $2$ (e.g. a global field of characteristic $2$) of degree and exponent equal to $8$.
    Let $K$ be the function field of the Severi-Brauer variety corresponding to $C\otimes_kC$ and set $D(8,2)=C\otimes_kK$. 
    Then $D(8,2)$ is a cyclic division algebra of degree $8$ and exponent $2$ over $K$, see \cite{SchBerg92}. 
    In \cite{Karp95}, it is shown that $D(8,2)$ does not decompose into a tensor product of $3$ quaternion algebras. 
    Hence by \Cref{mixed-biqud-non-decom-cyclic}, any mixed biquadratic and simple purely inseparable quartic field extension contained in $D(8,2)$ is not excellent.  
\end{exmp}

\medskip

\subsection*{Acknowledgement}
This work was done during some (online) meetings between the three authors when the first author was visiting the Laboratoire de Math\'ematiques de Lens in november 2023 in the framework of the project \emph{IEA of CNRS} between the Artois University and the Academic College of Tel-Aviv-Yaffo. 
The three authors are grateful for the support of both institutions and CNRS. 
The first author is supported by the 2020 PRIN (project \emph{Derived and underived algebraic stacks and applications}) from MIUR, and by research funds from Scuola Normale Superiore.

\end{document}